\documentclass{pspum-l}
\usepackage{amssymb}
\usepackage{enumerate}
\usepackage{color}

\newtheorem{thm}{Theorem}
\newtheorem{cor}[thm]{Corollary}
\newtheorem{lem}[thm]{Lemma}
\newtheorem{prop}[thm]{Proposition}
\newtheorem{rem}[thm]{Remark}
\newtheorem{ex}[thm]{Example}
\newtheorem{defi}[thm]{Definition}
\newtheorem{quest}[thm]{Question}

\begin{document}

\title{On a conjecture of A. Bikchentaev}

\author{F. A. Sukochev}

\address{School of Mathematics and Statistics, University of New South Wales, Sydney, NSW 2052, Australia }
\email{f.sukochev@unsw.edu.au}


\thanks{The author acknowledges support from the Australian Research Council.}


\begin{abstract} In \cite{bik1}, A. M. Bikchentaev conjectured that for positive
$\tau-$measurable operators $a$ and $b$ affiliated with an
arbitrary semifinite von Neumann algebra $\mathcal M$, the
operator $b^{1/2}ab^{1/2}$ is submajorized by the operator $ab$ in
the sense of Hardy-Littlewood. We prove this conjecture in full
generality and present a number of applications to fully symmetric
operator ideals, Golden-Thompson inequality and (singular) traces.
\end{abstract}

\maketitle

\section{Introduction and preliminaries}


In this paper we answer a question due to A. M. Bikchentaev (see \cite[p. 573, Conjecture A]{bik1}) in the affirmative (see also \cite{bik2,bik3}). To
formulate his conjecture, we need some notions from the theory of
noncommutative integration. For details on von Neumann algebra
theory, the reader is referred to e.g. \cite{Dix}, \cite{KR1, KR2}
or \cite{Tak}. General facts concerning measurable operators may
be found in \cite{Ne}, \cite{Se} (see also \cite[Chapter
IX]{Ta2} and the forthcoming book \cite{DPS}). For the convenience of the reader, some of the basic
definitions are recalled.

In what follows,  $H$ is a Hilbert space and $B(H)$ is the
$*$-algebra of all bounded linear operators on $H$, and
$\mathbf{1}$ is the identity operator on $H$. Let $\mathcal{M}$ be
a von Neumann algebra on $H$.

A linear operator $x:\mathfrak{D}\left( x\right) \rightarrow H $, where the domain $\mathfrak{D}\left( x\right) $ of $x$ is a linear subspace of $H$, is said to be {\it affiliated} with $\mathcal{M}$ if $yx\subseteq xy$ for all $y\in \mathcal{M}^{\prime }$. A linear
operator $x:\mathfrak{D}\left( x\right) \rightarrow H $ is termed {\it measurable} with respect to $\mathcal{M}$ if $x$ is closed, densely defined, affiliated with $\mathcal{M}$ and there exists a sequence $\left\{ p_n\right\}_{n=1}^{\infty}$ in $P\left(\mathcal{M}\right) $ such that $p_n\uparrow 1,$ $p_n(H)\subseteq\mathfrak{D}\left(x\right) $ and
$1-p_n$ is a finite projection (with respect to
$\mathcal{M}$) for all $n$. It should be noted that the condition
$p_{n}\left( H\right) \subseteq \mathfrak{D}\left( x\right) $
implies that $xp_{n}\in \mathcal{M}$. The collection of all
measurable operators with respect to $\mathcal{M}$ is denoted by
$S\left( \mathcal{M} \right) $, which is a unital $\ast $-algebra
with respect to
strong sums and products (denoted simply by $x+y$ and $xy$ for all $x,y\in S\left( \mathcal{M%
}\right) $).

Let $a$ be a self-adjoint operator affiliated with $\mathcal{M}$.
We denote its spectral measure by $\{e^a\}$. It is known that if $x$ is
a closed operator affiliated with $\mathcal{M}$ with the polar decomposition $x = u|x|,$ then $u\in\mathcal{M}$ and $e\in \mathcal{M}$ for all projections $e\in \{e^{|x|}\}$. Moreover,
$x\in S(\mathcal{M})$ if and only if $x$ is closed, densely
defined, affiliated with $\mathcal{M}$ and $e^{|x|}(\lambda,
\infty)$ is a finite projection for some $\lambda> 0$. It follows
immediately that in the case when $\mathcal{M}$ is a von Neumann
algebra of type $III$ or a type $I$ factor, we have
$S(\mathcal{M})= \mathcal{M}$. For type $II$ von Neumann algebras,
this is no longer true. From now on, let $\mathcal{M}$ be a
semifinite von Neumann algebra equipped with a faithful normal
semifinite trace $\tau$ (the reader unfamiliar with von Neumann
algebra theory can assume that $\mathcal{M}=B(H)$ and $\tau$ is
the standard trace on $B(H)$: the A. Bikchentaev's question
retains its interest also in this special case).

An operator $x\in S\left( \mathcal{M}\right) $ is called $\tau-$measurable if there exists a sequence $\left\{p_n\right\}_{n=1}^{\infty}$ in $P\left(\mathcal{M}\right)$ such that $p_n\uparrow 1,$ $p_n\left(H\right)\subseteq \mathfrak{D}\left(x\right)$ and $\tau(1-p_n)<\infty $ for all $n.$ The collection $S\left( \tau \right) $ of all $\tau $-measurable
operators is a unital $\ast $-subalgebra of $S\left(
\mathcal{M}\right) $ denoted by $S\left( \mathcal{M}, \tau\right)
$. It is well known that a linear operator $x$ belongs to $S\left(
\mathcal{M}, \tau\right) $ if and only if $x\in S(\mathcal{M})$
and there exists $\lambda>0$ such that $\tau(e^{|x|}(\lambda,
\infty))<\infty$. Alternatively, an unbounded operator $x$
affiliated with $\mathcal{M}$ is  $\tau-$measurable (see
\cite{FK}) if and only if
$$\tau\left(e^{|x|}(\frac1n,\infty)\right)=o(1),\quad n\to\infty.$$

Let $L_0$ be a space of Lebesgue measurable functions either on
$(0,1)$ or on $(0,\infty)$, or on $\mathbb{N}$ finite almost
everywhere (with identification $m-$a.e.). Here $m$ is Lebesgue
measure or else counting measure on $\mathbb{N}$. Define $S$ as
the subset of $L_0$ which consists of all functions $x$ such that
$m(\{|x|>s\})$ is finite for some $s.$

%

The notation $\mu(x)$ stands for the non-increasing right-continuous
rearrangement of $x\in S$ given by
$$\mu(t;x)=\inf\{s\geq0:\ m(\{|x|\geq s\})\leq t\}.$$
In the case when $x$ is a sequence we denote by $\mu(x)$ the usual
decreasing rearrangement of the sequence $|x|$.

%
%

Let a semifinite von Neumann  algebra $\mathcal M$ be equipped
with a faithful normal semi-finite trace $\tau$. Let $x\in
S(\mathcal{M},\tau)$. The generalized singular value function $\mu(x):t\rightarrow \mu(t;x)$ of
the operator $x$ is defined by setting
$$\mu(s;x)=\inf\{\|xp\|:\ p=p^*\in\mathcal{M}\mbox{ is a projection,}\ \tau(1-p)\leq s\}.$$
There exists an equivalent definition which involves the distribution function of the operator $x.$ For every self-adjoint operator $x\in
S(\mathcal{M},\tau),$ setting
$$d_x(t)=\tau(e^{x}(t,\infty)),\quad t>0,$$
we have (see e.g. \cite{FK})
$$\mu(t;x)=\inf\{s\geq0:\ d_{|x|}(s)\leq t\}.$$

Consider the algebra $\mathcal{M}=L^\infty(0,\infty)$ of all Lebesgue measurable essentially bounded functions on $(0,\infty).$ Algebra $\mathcal{M}$ can be seen as an abelian von Neumann algebra acting via multiplication on the Hilbert space $\mathcal{H}=L^2(0,\infty)$, with the trace given by integration with respect to Lebesgue measure $m.$ It is easy to see that the set of all measurable (respectively, $\tau$-measurable) operators affiliated with $\mathcal{M}$ can be identified with $S$ (respectively, with $L_0$). It should also be pointed out that the generalized singular value function $\mu(f)$ is precisely the decreasing rearrangement
$\mu(f)$ of the function $f$ defined above.

If $\mathcal{M}=B(H)$ (respectively, $l_\infty$) and $\tau$ is the standard trace ${\rm Tr}$ (respectively, the counting measure on $\mathbb{N}$), then it is not difficult to see that $S(\mathcal{M})=S(\mathcal{M},\tau)=\mathcal{M}.$ In this case, for $x\in S(\mathcal{M},\tau)$ we have
$$\mu(n;x)=\mu(t;x),\quad t\in[n,n+1),\quad  n\geq0.$$
The sequence $\{\mu(n;x)\}_{n\geq0}$ is just the sequence of singular values of the operator $x.$

Let $a,b\in S(\mathcal{M},\tau).$ We say that $b$ is submajorized
by $a$ in the sense of Hardy-Littlewood-Polya if and only if
$$\int_0^t\mu(s;b)ds\leq\int_0^t\mu(s;a)ds,\quad\forall t>0.$$
In this case, we write $b\prec\prec a.$ Observe that $b\prec\prec a$ if and only if $\mu(b)\prec\prec\mu(a).$ Sometimes, we also write $a\prec\prec f$ instead of $\mu(a)\prec\prec\mu(f).$

In the special case when $a$ and $b$ are {\it positive}
self-adjoint operators from $S(\mathcal{M},\tau)$ the following
question was asked in \cite{bik1}.

\begin{quest}\label{mainquestion} Let $a$ and $b\ge 0$ be self-adjoint $\tau-$measurable
operators affiliated with $\mathcal{M}$.
Is it necessarily true that
$$b^{1/2}ab^{1/2}\prec\prec ab?$$
\end{quest}

If $\mathcal{M}$ is a matrix algebra, then the positive answer to Question \ref{mainquestion} may be inferred from \cite{kosaki}. The main objective of the present article is to provide a different (stronger and more general) approach to Question 1. Our method allows us to produce a number of applications.

Observe that the inequality $\mu(b^{1/2}ab^{1/2})\leq\mu(ab)$ fails even for the case $\mathcal{M}=M_2(\mathbb{C})$ (see Remark 2, p. 575 of \cite{bik2}).

The author thanks A. Bikchentaev for drawing his attention to this problem and additional references and D. Zanin and B. de Pagter for a number of insightful comments which improved the article. Some results of this article have been announced in \cite{S}.

\section{The main result}

The gist of our approach to Question \ref{mainquestion} is
contained in Lemmas \ref{invertible case} and \ref{bounded case}
below. In the proofs we use two properties of singular value
functions (see \eqref{eq1} and \eqref{eq2} below  (see also
\cite[Lemma 4.1]{FK}, \cite{CS} and \cite[Proposition
3.10]{DDP3}). For simplicity of exposition, we shall assume that
$\mathcal{M}$ is an atomless von Neumann algebra. Indeed, this is
done by a standard trick consisting in considering a von Neumann
tensor product $\mathcal{M}\otimes L_\infty(0,1)$ (see details in
e.g. \cite[pp. 574-575]{bik1}).

Let $L_1$ and $L_\infty$ be Lebesgue spaces on
$(0,\tau(\mathbf{1}))$.

Let $a\in S(\mathcal{M},\tau)$. We say that $a\in
L_1(\mathcal{M},\tau)$ if and only if $\mu(a)\in L_1(0,\infty)$.
It is well-known that
$$
\|a\|_1:=\|\mu(a)\|_{L_1}
$$
is a Banach norm on $L_1(\mathcal{M},\tau)$. Similarly, we say
that $a\in (L_1+L_{\infty})(\mathcal{M},\tau)$ if and only if
$\mu(a)\in(L_1+L_{\infty})(0,\infty)$. Here, we identify
$\mathcal{M}$ with $L_\infty (\mathcal{M},\tau)$. The space
$(L_1+L_{\infty})(\mathcal{M},\tau)$ can be also viewed as a sum
of Banach spaces $L_1(\mathcal{M},\tau)$ and $L_\infty
(\mathcal{M},\tau)$ (the latter space is equipped with the uniform
norm, which we denote simply by $\|\cdot\|$).

For all $a,c\in(L_1+L_{\infty})(\mathcal{M}),$ we have (see \cite{FK,CS})
\begin{equation}\label{eq1}
\mu(ac)\prec\prec\mu(a)\mu(c).
\end{equation}

\begin{lem}\label{Ben} If $a\in(L_1+L_{\infty})(\mathcal{M},\tau),$
then
\begin{equation}\label{eq2}
\int_0^t\mu(s;a)ds=\sup\{|\tau(ac)|:\ |\tau(s(c))|\leq t,\
\|c\|\leq 1\},
\end{equation}
where $s (c)$ denotes the support projection of the operator c.
\end{lem}
\begin{proof}
Recall that if $x\in S((\mathcal{M},\tau)$, then
$$\tau (s (x)) =\inf \{ s\ge 0 : \mu(s; x) = 0\}.$$
Therefore, if $c\in \mathcal{M}$ satisfies
$\|c\|\leq 1$ and $|\tau(s(c))|\leq t$, then $s(ac)\leq s(c)$ and
so
$$
|\tau(ac)|\leq \tau(|ac|)=\int_0^\infty \mu(s;ac)ds=\int_0^t
\mu(s;ac)ds.
$$
For the converse inequality, first recall the following fact (see \cite{FK,DPS}): if
$x\in S((\mathcal{M},\tau)$, then (under the assumption that there are no atoms in $\mathcal{M}$) we have
$$
\int_0^t \mu(s;x)ds=\sup\{\tau(p|x|p):\ p=p^*=p^2\in\mathcal{M},\
\tau(p)\leq t\}.
$$
If $a\in(L_1+L_{\infty})(\mathcal{M},\tau),$ and $\tau (p) \leq t$
then $pap$, $ap \in L_1(\mathcal{M},\tau)$ and so,
$\tau(p|a|p)=\tau(|a|p)$. Hence,
$$
\int_0^t \mu(s;a)ds=\sup\{\tau(|a|p):\ p=p^*=p^2\in\mathcal{M},\
\tau(p)\leq t\}
$$
for all $a\in(L_1+L_{\infty})(\mathcal{M},\tau)$.

If $p=p^*=p^2\in\mathcal{M}$ with $\tau(p)\leq t$, then $|a| p =
v^*ap$ (where $a = v|a|$ is the polar decomposition of a) and so,
$\tau(|a| p) =\tau( v^*ap)=\tau (apv^*)$. Setting $c = pv^*$, it
follows that $\mu (c)\leq \mu(p)$ and so, $\mu (s; c) = 0$ for all
$s\ge\tau (p)$, that is, $\tau(s (c)) \leq\tau (p) \leq $t. The
result of the lemma follows.
\end{proof}

The following remark is well-known (and trivial) and stated here
for convenience of the reader.

\begin{rem}\label{trivial} Let $b\in\mathcal{M}.$ The mapping $z\to e^{zb}$ takes its values in $\mathcal{M}$ and is holomorphic on $\mathbb{C}.$
\end{rem}

\begin{lem}\label{invertible case} For any self-adjoint operators
$a,b\in\mathcal{M}$ and every $\theta\in(0,1)$, we have
$$e^{\theta b}ae^{(1-\theta)b}\prec\prec ae^b.$$
\end{lem}
\begin{proof} Appealing to Remark \ref{trivial}, we see that the mapping $z\to e^{zb}ae^{(1-z)b}$ takes values in $\mathcal{M}$ and is holomorphic on $\mathbb{C}.$

Fix an operator $c\in\mathcal{M}$ such that $\|c\|\leq1$ and
$\tau(s(c))\leq t$. The $\mathbb{C}-$valued function $F:z\to\tau(e^{bz}ae^{b(1-z)}c)$
is also holomorphic on $\mathbb{C}.$ For every $0\leq\Re z\leq 1,$ it
follows from \eqref{eq1} that
$$|F(z)|\leq t\|e^{zb}ae^{(1-z)b}\|\leq t\|e^{zb}\|\cdot\|a\|\cdot\|e^{(1-z)b}\|\leq t\|a\|e^{2\|b\|}.$$
Hence, $F$ is a bounded function in the strip $0\leq\Re z\leq 1.$ Since $F$ is holomorphic on $\mathbb{C},$
it follows that $F$ is continuous on the boundary of the strip $0\leq\Re z\leq 1.$
By Hadamard three-lines theorem (see e.g. \cite[p. 33-34]{rs}), we have
$$|F(z)|\leq\sup_{y\in\mathbb{R}}\max\{|F(iy)|,|F(1+iy)|\}.$$

To estimate $|F(iy)|,$ we argue as follows:
$$|F(iy)|=|\tau(e^{iyb}ae^be^{-iyb}c)|\leq
\tau(|ae^be^{-iyb}ce^{iyb}|)=\tau(|ae^bd|),
$$
where $d=e^{-iyb}ce^{iyb}$. Since $\mu(ae^bd)\leq
\|ae^b\|\mu(d)=\|ae^b\|\mu(c)$ it is clear that $\mu(ae^bd;s)=0$
for $s\ge t$. Hence,
$$|\tau(ae^bd)|=\int_0^\infty\mu(ae^bd;s)ds\leq\int_0^t\mu(ae^bd;s)ds\leq\int_0^t\mu(ae^b;s)ds.$$
Similarly,
$$|F(1+iy)|=|\tau(e^ba)|\leq\int_0^t\mu(s;e^ba)ds$$
and, appealing to the assumption $a=a^*,$ we may conclude that
$$|F(z)|=|\tau(e^{bz}ae^{b(1-z)}c)|\leq\int_0^t\mu(s;ae^b)ds$$
for all $z\in\mathbb{C}$ with $0\leq\Re z\leq 1.$ Since this holds for all operators $c\in\mathcal{M}$ such that $|\tau(s(c))|\leq t$ and $\|c\|\leq1$, we obtain from Lemma \ref{Ben} that the estimate
$$\int_0^t\mu(s;e^{zb}ae^{(1-z)b})ds\leq \int_0^t\mu(s;ae^b)ds$$
holds for all $z\in \mathbb{C}$ with $0\leq\Re z\leq 1.$ Setting $z=\theta\in(0,1),$ we conclude the proof.
\end{proof}

Observe that the assumption that $a\in\mathcal{M}$ is a self-adjoint operator was crucially used in the proof above, where we concluded that $\mu(e^ba)=\mu(ae^b).$ In fact, the assertion of
the above lemma does not hold for a non-self-adjoint operator $a.$

\begin{ex}\label{tr} Let $M_2(\mathbb{C})$ be the von Neumann algebra of all $2\times 2$ matrices. There exist $a,b\in M_2(\mathbb{C})$ such $b=b^*$ and such that the inequality
$$e^{b/2}ae^{b/2}\prec\prec ae^b$$
fails.
\end{ex}
\begin{proof} Let $\lambda,\mu\in\mathbb{R}.$ We set
$$a=\begin{pmatrix}
0&1\\
0&0
\end{pmatrix},
b=\begin{pmatrix}
\lambda&0\\
0&\mu
\end{pmatrix}.
$$
A direct computation yields $ae^b=e^{\mu}a$ and $e^{b/2}ae^{b/2}=e^{(\lambda+\mu)/2}a.$ Setting $\lambda>\mu,$ we obtain the assertion.
\end{proof}

However, a quick analysis of the proof of Lemma \ref{invertible case} yields a following strengthening.

\begin{lem}\label{invertible case2} Let $a,b\in\mathcal{M}$. If $b=b^*,$ then
$$e^{\theta b}ae^{(1-\theta)b}\prec\prec \max \{\mu(ae^b),\ \mu (e^ba)\}$$ for every $\theta\in(0,1).$
\end{lem}

For every $\varepsilon,\delta>0,$ we define the set
$$V(\varepsilon,\delta)=\{x\in S(\mathcal{M},\tau):\ \exists p=p^2=p^*\in\mathcal{M}\mbox{ such that }
\|x(1-p)\|\leq\varepsilon,\ \tau(p)\leq\delta\}.$$ The topology
generated by the sets $V(\varepsilon,\delta),$
$\varepsilon,\delta>0,$ is called a measure topology.

The following assertion is well-known. We incorporate the proof
for convenience of the reader.

\begin{lem}\label{fk1 lemma} Let $x_n,x\in S(\mathcal{M},\tau)$ be such that $x_n\to x$ in measure topology.
It follows that $\mu(x_n)\to\mu(x)$ almost everywhere.
\end{lem}
\begin{proof} Let $t$ be the continuity point of $\mu(x)$.
Fix $\varepsilon>0$ and select $\delta>0$ such that
$|\mu(t;x)-\mu(t\pm\delta;x)|\leq\varepsilon$. Since $x_n-x\to 0$
in measure, it follows that $x_n-x\in V(\varepsilon,\delta)$ for
every $n\geq N.$ Thus, $\mu(\delta;x_n-x)\leq\varepsilon.$

We have
$$\mu(t;x)\leq\mu(t+\delta;x)+\varepsilon\leq\varepsilon+\mu(t;x_n)+\mu(\delta;x_n-x)\leq 2\varepsilon+\mu(t;x_n)$$
and
$$\mu(t;x_n)\leq\mu(t-\delta;x)+\mu(\delta;x-x_n)\leq2\varepsilon+\mu(t;x).$$
Thus, $\mu(t;x_n)\to\mu(t;x).$ The assertion follows from the fact that $\mu(x)$ is almost everywhere continuous.
\end{proof}

\begin{lem}\label{bounded case} Let $a,b\in\mathcal{M}$

\begin{enumerate}[(i)]
\item If $a$ is self adjoint and $b\ge 0$, then
$$b^{\theta}ab^{1-\theta}\prec\prec ab$$
for every $\theta\in(0,1).$

\item  If $a$ is an arbitrary operator and $b\ge 0$, then we have
$$b^{\theta}ab^{1-\theta}\prec\prec \max \{\mu(ab),\ \mu (ba)\}$$ for every $\theta\in(0,1).$
\end{enumerate}
\end{lem}
\begin{proof} We shall prove only the first assertion (the proof
of the second is the same via Lemma \ref{invertible case2}).

Suppose first that $b$ is a positive invertible operator from
$\mathcal{M}$. Then $\log(b)$ is a self-adjoint operator from
$\mathcal{M}$ and applying Lemma \ref{invertible case} to the
bounded operators $a$ and $\log(b),$ we obtain the assertion.

In general case, fix $n\in\mathbb{N},$ and consider the operator $b_n:=b+\frac1n$ which is obviously invertible. It follows from the first part of the proof that
$$\int_0^t\mu(s;b_n^{\theta}ab_n^{1-\theta})ds\leq\int_0^t\mu(s;ab_n)ds.$$
By Lemma \ref{fk1 lemma}
$$\mu(b_n^{\theta}ab_n^{1-\theta})\to\mu(b^{\theta}ab^{1-\theta}),
\quad\mu(ab_n)\to\mu(ab)$$ almost everywhere. Since the functions
$\mu(b^{\theta}ab^{1-\theta})$ and $\mu(ab)$ are uniformly
bounded, we easily infer from here that for every $t\ge 0$, we
have
$$\int_0^t\mu(s;b^{\theta}ab^{1-\theta})ds\leq\int_0^t\mu(s;ab)ds.$$
\end{proof}

The resolution of Question \ref{mainquestion} is contained in the
first part of the theorem below. Observe that only the case
$ab,ba\in(L_1+L_{\infty})(\mathcal{M},\tau)$ should be treated.
Indeed, if the latter assumption does not hold then the answer to
Question \ref{mainquestion} is trivially affirmative.

\begin{thm}\label{bik th} Let $\mathcal{M}$ be a von Neumann algebra and
let $a,b\in S(\mathcal{M},\tau)$ be such operators that $b\ge 0$
and $ab\in(L_1+L_{\infty})(\mathcal{M},\tau).$
\begin{enumerate}[(i)]
\item If $a$ is self-adjoint, then
$$b^{\theta}ab^{1-\theta}\prec\prec ab$$
for every $\theta\in(0,1)$.
\item  If $a$ is an arbitrary operator, then we have
$$b^{\theta}ab^{1-\theta}\prec\prec \max \{\mu(ab),\ \mu (ba)\}$$
for every $\theta\in(0,1)$.
\end{enumerate}
\end{thm}
\begin{proof} We prove the second assertion. Let $p_n=ae^{|a|}[0,n)$ and let $q_n=e^{b}[0,n).$ The operators $ap_n$ and $q_nbq_n=bq_n$ are bounded and evidently,
$ap_n\to a$ and $bq_n\to b$ in measure as $n\to\infty.$ Hence, $(bq_n)^{\theta}\to b^{\theta}$ in measure (see e.g. \cite{Tik}) and, therefore,
$$(bq_n)^{\theta}(ap_n)(bq_n)^{1-\theta}\to b^{\theta}ab^{1-\theta}$$
in measure. By Lemma \ref{fk1 lemma}, we have
\begin{equation}\label{eds}
\mu((bq_n)^{\theta}(ap_n)(bq_n)^{1-\theta})\to\mu(b^{\theta}ab^{1-\theta})
\end{equation}
almost everywhere. It follows now from Fatou lemma that
$$\int_0^t\mu(s;b^{\theta}ab^{1-\theta})ds\leq\liminf_{n\to\infty}\int_0^t\mu(s;(bq_n)^{\theta}(ap_n)(bq_n)^{1-\theta})ds.$$
By Lemma \ref{bounded case}, we have
$$\int_0^t\mu(s;b^{\theta}ab^{1-\theta})ds\leq\liminf_{n\to\infty}\int_0^t\max\{\mu(s;(bq_n)(ap_n)),\mu(s;(ap_n)(bq_n))\}ds.$$
Since $|ad|=||a|d|$ for all operators $a,d\in S(\mathcal{M},\tau),$ it follows that
$$\mu((ap_n)(bq_n))=\mu(|a|p_n(bq_n))=\mu(p_n(|a|b)q_n)\leq\mu(|a|b)=\mu(ab).$$
Also, we have
$$\mu((bq_n)(ap_n))=\mu(q_n(ba)p_n)\leq\mu(ba).$$
The assertion follows immediately.
\end{proof}

The result of Theorem \ref{bik th} above extends and complements \cite[Theorems 1 and 2]{bik1}, \cite[Theorem 3 and Corollary 4]{bik2}, \cite[Proposition 3.4]{DDP3}. More details are given in the next section.

We end this section with one more extension of Lemma \ref{invertible case}.

\begin{prop}\label{invertible case3} For any self-adjoint operators
$a,b\in S(\mathcal{M},\tau)$ and every $\theta\in(0,1)$, we have
$$e^{\theta b}e^ae^{(1-\theta)b}\prec\prec e^ae^b.$$
\end{prop}
\begin{proof} By \cite[Lemma 3.1]{PS}, we have $e^a,e^b\in
S(\mathcal{M},\tau)$. It is sufficient to prove the assertion only
for the case $e^ae^b\in(L_1+L_{\infty})(\mathcal{M},\tau)$.
Introducing projections $p_n:=e^{|a|}[0,n)$, $q_n:=e^{|b|}[0,n)$,
and operators $a_n:=ap_n, b_n:=bq_n$ we obtain from Lemma
\ref{invertible case} that
$$e^{\theta b_n}e^{a_n}e^{(1-\theta)b_n}\prec\prec e^{a_n}e^{b_n},\ n\ge 1.$$
The same argument as in the proof of Theorem \ref{bik th}
completes the proof.
\end{proof}

\section{Applications to ideals in $S(\mathcal{M},\tau)$}

The best known examples of normed $\mathcal{M}$-bimodules of
$S(\mathcal{M},\tau)$ are given by the so-called symmetric
operator spaces (see e.g. \cite{DDP0, SC, KS,DPS}). We briefly recall
relevant definitions (for more detailed information we refer to
\cite{KS} and references therein).

Let $E$  be a Banach space of real-valued Lebesgue measurable
functions either on $(0,1)$ or $(0,\infty)$ (with identification
$m-$a.e.) or on $\mathbb{N}$. The space $E$ is said to be {\it
absolutely solid} if $x\in E$ and $|y|\leq |x|$, $y\in L_0$
implies that $y\in E$ and $||y||_E\leq||x||_E.$

The absolutely solid space $E\subseteq S$ is said to be {\it
symmetric} if for every $x\in E$ and every $y$ the assumption
$\mu(y)=\mu(x)$ implies that $y\in E$ and $||y||_E=||x||_E$ (see e.g.
\cite{KPS}).

If $E=E(0,1)$ is a symmetric space on $(0,1),$ then
$$L_{\infty}\subseteq E\subseteq L_1.$$

If $E=E(0,\infty)$ is a symmetric space on $(0,\infty),$ then
$$L_1\cap L_{\infty}\subseteq E\subseteq L_1+L_{\infty}.$$

If $E=E(\mathbb{N})$ is a symmetric space on $\mathbb{N},$ then
$$l_1\subseteq E\subseteq l_{\infty},$$
where $l_1$ and $l_\infty$ are classical spaces of all
absolutely summable and bounded sequences respectively.

\begin{defi}\label{opspace}
Let $\mathcal{E}$ be a linear subset in $S({\mathcal{M}, \tau})$
equipped with a norm $\|\cdot\|_{\mathcal{E}}$. We say that
$\mathcal{E}$ is a \textit{symmetric operator space} (on
$\mathcal{M}$, or in $S({\mathcal{M}, \tau})$) if $x\in
\mathcal{E}$ and every $y\in S({\mathcal{M}, \tau})$ the
assumption $\mu(y)\leq \mu(x)$ implies that $y\in \mathcal{E}$ and
$\|y\|_\mathcal{E}\leq \|x\|_\mathcal{E}$.
\end{defi}

The fact that every symmetric operator space $\mathcal{E}$ is (an
absolutely solid) $\mathcal{M}$-bimodule of $S\left(\mathcal{M},
\tau \right)$ is well known (see e.g. \cite{SC, KS} and references
therein). In the special case, when $\mathcal{M}=B(H)$ and $\tau$
is a standard trace ${\rm Tr}$, the notion of symmetric operator
space introduced in Definition \ref{opspace} coincides with the
notion of symmetric operator ideal \cite{GK1, GK2, Schatten,
Simon}.

\begin{defi}\label{opideal}
A linear subspace ~$\mathcal I$~ in the von Neumann algebra
$\mathcal M$ equipped with a norm $\|\cdot\|_{\mathcal I}$ is said
to be {\it a symmetric operator ideal} if
\begin{enumerate}
\item $\| S\|_{\mathcal I}\geq \| S\|$ for all $S\in \mathcal I.$
\item $\| S^*\|_{\mathcal I} = \| S\|_{\mathcal I}$ for all $S\in \mathcal I.$
\item $\| ASB\|_{\mathcal I}\leq \| A\| \:\| S\|_{\mathcal I}\| B\|$
for all $S\in \mathcal I$, $A,B\in \mathcal M$.
\end{enumerate}
\end{defi}

There exists a strong connection between symmetric function and
operator spaces recently exposed in \cite{KS} (see earlier results
in \cite{Schatten, GK1, GK2, Simon}).

Let $E$ be a symmetric function space on the interval $(0,1)$
(respectively, on the semi-axis) and let $\mathcal{M}$ be a type
$II_1$ (respectively, $II_{\infty}$) von Neumann algebra. Define
$$E(\mathcal{M},\tau):=\{x\in S(\mathcal{M},\tau):\ \mu(x)\in E\},
\ \|x\|_{E(\mathcal{M},\tau)}:=\|\mu(x)\|_E.$$ Main results of
\cite{KS} assert that $(E(\mathcal{M},\tau),
\|\cdot\|_{E(\mathcal{M},\tau)})$ is a symmetric operator space.
Similarly, if
$E=E(\mathbb{N})$ is a symmetric sequence space on $\mathbb{N},$
and the algebra $\mathcal {M}$ is a type $I$ factor with standard
trace, then (see \cite{KS}) setting
$$\mathcal{E}:=\{x\in \mathcal{M}:\ (\mu(n;x))_{n\geq0}\in E\},
\ \|x\|_{\mathcal{E}}:=\|(\mu(n;x))_{n\geq0}\|_E$$ yields a
symmetric operator ideal. Conversely, every symmetric operator
ideal $\mathcal{E}$ in $\mathcal {M}$ defines a unique symmetric
sequence space $E=E(\mathbb{N})$ by setting
$$
E:=\{a=(a_n)_{n\geq0}\in l_\infty:\
\mu(a)=(\mu(n;x))_{n\geq0}\ \text{for some}\ x\in \mathcal{E}\},\
\|a\|_E:=\|x\|_{\mathcal{E}}
$$

Finally, a symmetric space $E(\mathcal{M},\tau)$ is called fully
symmetric if for every $a\in E(\mathcal{M},\tau)$ and every
$b\in(L_1+L_{\infty})(\mathcal{M})$ with $b\prec\prec a,$ we have
$b\in E(\mathcal{M},\tau)$ and $\|b\|_E\leq\|a\|_E$. The following
result now follows immediately from Theorem \ref{bik th}.

\begin{cor}\label{firstcor} Let $E$ be a fully symmetric function space on
$(0,\tau(\mathbf{1}))$. If $a,b\in S(\mathcal{M},\tau)$ are such
operators that $b\geq0$ and $ab, ba\in E(\mathcal{M},\tau)$, then
$b^{\theta}ab^{1-\theta}\in E(\mathcal{M},\tau)$ for every
$\theta\in(0,1)$. In particular, if $a$ is self-adjoint, then
$$ab\in E(\mathcal{M},\tau)\Longrightarrow b^{\theta}ab^{1-\theta}\in
E(\mathcal{M},\tau),\ \forall \theta\in(0,1)$$ and
$\|b^{\theta}ab^{1-\theta}\|_{E(\mathcal{M},\tau)}\leq
\|ab\|_{E(\mathcal{M},\tau)}$.
\end{cor}

We shall now present some variation of the result above. For simplicity of the exposition, we shall do so for fully symmetric sequence spaces $E$ and for symmetric operator ideals $\mathcal{E}$, although  all arguments below can be repeated also for general semifinite factors.

\begin{cor}\label{secondcor} Fix a fully symmetric operator ideal $\mathcal{E}$ and
let $a, b_0, b_1\in B(H)$, $b_0, b_1\ge 0$, $\theta\in(0,1)$.

\begin{enumerate}[(i)]
\item If $ab_0, b_1a\in \mathcal{E}$, then $b_0^{\theta}ab_1^{1-\theta},b_1^{\theta}ab_0^{1-\theta}\in \mathcal{E}$ and
$$b_1^{\theta}ab_0^{1-\theta}\prec\prec \max \{\mu(ab_0),\ \mu(b_1a)\}.$$
In particular, $b_1^{\theta}ab_0^{1-\theta}\in\mathcal{E}.$
\item If $b_0a, ab_1\in \mathcal{E}$, then
$$\mu(b_0^{\theta}ab_1^{1-\theta})\oplus\mu(b_1^{\theta}a^*b_0^{1-\theta})\prec\prec \mu(ab_1)\oplus\mu (b_0a).$$
In particular, if $a=a^*$ and $\theta=1/2$, we have
$$\sigma_2(\mu(b_0^{1/2}ab_1^{1/2})\prec\prec \mu(ab_1)\oplus\mu(ab_0).$$
Here $\sigma_2(a_0,a_1,\cdots)=(a_0,a_0,a_1,a_1,\cdots).$
\end{enumerate}
\end{cor}
\begin{proof} (i) In $B(H\oplus H)$ consider the
following operators
$$
\mathbf{a}=\begin{pmatrix}
0&0\\
a&0
\end{pmatrix},\quad
\mathbf{b}=\begin{pmatrix}
b_0&0\\
0&b_1
\end{pmatrix}.
$$
We obviously have $\mathbf{b}\ge 0$ and $\mu
(\mathbf{a}\mathbf{b})=\mu (ab_0)$, $\mu
(\mathbf{b}\mathbf{a})=\mu (b_1a)$ and therefore
$\mathbf{a}\mathbf{b}, \mathbf{b}\mathbf{a}\in \mathcal{E}$.
Applying now Theorem \ref{bik th} and Corollary \ref{firstcor} we
arrive at
$$
\begin{pmatrix}
0&0\\
b_1^{\theta}ab_0^{1-\theta}&0
\end{pmatrix}
\prec\prec \max \left\{\mu\left(\begin{pmatrix}
0&0\\
ab_0&0
\end{pmatrix}\right),\ \mu \left(\begin{pmatrix}
0&0\\
b_1a&0
\end{pmatrix}\right)\right\},$$
which is the assertion.

(ii) Consider $\mathbf{b}$ as above and set
$$
\mathbf{a}:=\begin{pmatrix}
0&a\\
a^*&0
\end{pmatrix}
$$
Observe that $\mathbf{a}$ is self-adjoint and that the assumption
guarantees $\mathbf{a}\mathbf{b}\in \mathcal{E}$. Thus, by Theorem
\ref{bik th} we have
$$\mathbf{b}^{\theta}\mathbf{a}\mathbf{b}^{1-\theta}\prec\prec
\mathbf{a}\mathbf{b},
$$
that is
$$
\begin{pmatrix}
0&b_0^{\theta}ab_1^{1-\theta}\\
b_1^{\theta}a^*b_0^{1-\theta}&0
\end{pmatrix}
\prec\prec
\begin{pmatrix}
0&ab_1\\
a^*b_0&0
\end{pmatrix}
$$
which is equivalent to the first assertion. The last assertion in (ii) is trivial.
\end{proof}

The following lemma extends result of \cite[Lemma 27]{cps} and \cite[Lemma 10]{ps}.

\begin{lem} Let $\mathcal{E}$ be a fully symmetric operator ideal. If $a,b_0,b_1\in B(H)$ are such that $b_0,b_1\geq0$ and $b_0a,ab_1\in\mathcal{E},$ then
$$\|b_0^{\theta}ab_1^{1-\theta}\|_{\mathcal{E}}\leq\|b_0a\|_{\mathcal{E}}^{\theta}\|ab_1\|_{\mathcal{E}}^{1-\theta},\quad 0<\theta<1.$$
\end{lem}
\begin{proof} Setting $F(z)=\tau(b_0^zab_1^{1-z}c)$ and repeating the argument in Lemma \ref{invertible case}, we obtain
$$\int_0^t\mu(s;b_0^{\theta}ab_1^{1-\theta})ds\leq(\int_0^t\mu(s;b_0a)ds)^{\theta}(\int_0^t\mu(s;ab_1)ds)^{1-\theta}.$$
Using elementary inequality $\alpha^{\theta}\beta^{1-\theta}\leq\theta\alpha+(1-\theta)\beta,$ we obtain, for every $\lambda>0,$
$$\int_0^t\mu(s;b_0^{\theta}ab_1^{1-\theta})ds\leq(\int_0^t\mu(s;\lambda^{1-\theta}b_0a)ds)^{\theta}(\int_0^t\mu(s;\lambda^{-\theta}ab_1)ds)^{1-\theta}\leq$$
$$\leq\theta\lambda^{1-\theta}\int_0^t\mu(s;b_0a)ds+(1-\theta)\lambda^{-\theta}\int_0^t\mu(s;ab_1)ds.$$
Since the ideal $\mathcal{E}$ is fully symmetric, it follows that
$$\|b_0^{\theta}ab_1^{1-\theta}\|_{\mathcal{E}}\leq\theta\lambda^{1-\theta}\|b_0a\|_{\mathcal{E}}+(1-\theta)\lambda^{-\theta}\|ab_1\|_{\mathcal{E}}.$$
The assertion follows now by setting
$$\lambda=\|ab_1\|_{\mathcal{E}}\cdot\|b_0a\|_{\mathcal{E}}^{-1}.$$
\end{proof}

Recall that the set
$${L}_p(\mathcal{M},\tau)=\left\{x\in S(\mathcal{M},\tau):\ \tau(|x|^p)<\infty\right\}$$
equipped  with a standard norm
$$\|x\|_p:=\tau(|x|^p)^{1/p}$$
is the noncommutative $L_p$-space associated with
$(\mathcal{M},\tau)$ for every $1\leq p<\infty.$ In the type $I$ factor setting these are the usual Schatten-von Neumann ideals \cite{GK1, GK2,Schatten, Simon}. The following corollary follows immediately from the above result.

\begin{cor}\label{thirdcor} Let $\mathcal{M}$ be a semifinite factor and $a,b_0,b_1\in S(\mathcal{M},\tau)$ be self-adjoint operators such that $b_0,b_1\geq0$ and
such that $ab_0, ab_1\in L_p(\mathcal{M},\tau).$ We have $b_0^{1/2}ab_1^{1/2}\in L_p(\mathcal{M},\tau)$ and
$$2\|b_0^{1/2}ab_1^{1/2}\|_p^p\leq\|ab_1\|_p^p +\|b_0a\|_p^p.$$
\end{cor}

For detailed exposition of (generalized) Golden-Thompson inequality and for further references we refer to \cite{Simon}. The following result now follows immediately from Proposition \ref{invertible case3} and the definition of a fully symmetric space.

\begin{prop}\label{Golden-Thompson} Let $E$ be a fully symmetric function space on $(0,\tau(\mathbf{1}))$.
For any self-adjoint operators
$a,b\in S(\mathcal{M},\tau)$ and every $\theta\in(0,1)$, we have
$$\|e^{\theta b}e^ae^{(1-\theta)b}\|_{E(\mathcal{M},\tau)}\leq\| e^ae^b\|_{E(\mathcal{M},\tau)}.$$
\end{prop}

We shall complete this section with a complement to \cite[Theorem
8.3]{Simon}. According to that theorem for self-adjoint operators
$a$ and $b$ and for all $1 \leq p \leq \infty$ we have
$\|e^{a+b}\|_p \leq \|e^{a/2}e^be^{a/2}\|_p$ and for $2 \leq p
\leq\infty$ we have $\|e^{a+b}\|_p \leq \|e^{a}e^b\|_p$. We claim
that the latter estimate holds for all $1 \leq p \leq \infty$.
Indeed, this follows from a combination of the former estimate and
Proposition \ref{Golden-Thompson}.

\section{An application to traces}

Let $\mathcal{E}$ be a (fully) symmetric operator ideal. A linear functional $\varphi$ on $\mathcal{E}$ is said to be a trace if $\varphi(ab)=\varphi(ba)$ for every $a\in \mathcal{E}$ and
every $b\in B(H).$ A complete characterization of symmetric operator ideals which admit a nontrivial trace has been recently given in \cite{SZcrelle}.

For a compact operator $x\in B(H),$ the symbol $\Lambda(x)$ stands for the {\it set} of all sequences of eigenvalues of $x$ counted with algebraic multiplicities and ordered by the inequality $|\lambda_{n+1}(x)|\leq|\lambda_n(x)|.$ The following assertion is a particular case of Theorem 3.10.3 in \cite{BirmanSolomjak}.

\begin{thm}\label{spectral lemma} If $a,b\in B(H)$ and if $a$ is a compact operator, then $\Lambda(ab)=\Lambda(ba).$
\end{thm}

The assertion of the previous theorem fails without the assumption of compactness.

\begin{ex} There exist bounded operators $a,b\in B(H)$ such that $ab$ is compact while $ba$ is not.
\end{ex}
\begin{proof} Fix an infinite projection $p$ such that $1-p$ is also an infinite projection. Thus, projections $p$ and $1-p$ are equivalent in $B(H).$ Select a partial isometry $u$ such that $uu^*=p$ and $u^*u=1-p.$ We have $|up|^2=pu^*\cdot up=p(1-p)p=0$. Hence, $up=0$ and $pu=(uu^*)u=u(1-p)=u-up=u.$ Setting $a=u$ and $b=p,$ we are done.
\end{proof}

The following fundamental result will first appear in \cite{KLPS},
though it is essentially proved in \cite{Kalton}.

\begin{thm}\label{Kalton} Let $\mathcal{E}$ be a symmetric operator ideal. For every $a\in \mathcal{E}$ and for every trace
$\varphi$ on $\mathcal{E}$, we have
$$\varphi(a)=\varphi({\rm diag}(\lambda(a))),$$
where ${\rm diag}(\lambda(a))$ is a diagonal matrix corresponding to any sequence $\lambda(a)\in \Lambda(a).$
\end{thm}

The following theorem is the main result of this section. It is new even in the case when $\varphi$ is the standard trace ${\rm Tr}$ on $B(H).$ For the special cases of this theorem for $\theta=1/2,$ we refer to \cite{DDP3,FK,kosaki}.

\begin{thm} Let $\mathcal{E}$ be a fully symmetric operator ideal and let $a,b\in B(H),$ $b\ge 0,$ be such that $a$ is compact and $ab,ba\in \mathcal{E}$. For every trace $\varphi$ on
$\mathcal{E}$, we have
$$\varphi(ab)=\varphi(ba)=\varphi(b^{1-\theta}ab^{\theta}).$$
\end{thm}
\begin{proof} By Theorem \ref{bik th}, we have
$b^{1-\theta}ab^{\theta}\prec\prec\max\{\mu(ab),\mu(ba)\}$ and, therefore, $b^{1-\theta}ab^{\theta}\in\mathcal{E}.$ By Theorem
\ref{spectral lemma}, we have
$$\Lambda(ab)=\Lambda(ba)=\Lambda(b^{1-\theta}ab^{\theta}).$$
The assertion follows now from Theorem \ref{Kalton}.
\end{proof}

\end{document}